\documentclass[twoside,10pt,leqno, A4]{amsart}
\usepackage{amsfonts}
\usepackage{amsmath}
\usepackage{amscd}
\usepackage{amssymb}
\usepackage{amsthm}
\usepackage{amsrefs}
\usepackage{latexsym}
\usepackage{mathrsfs}
\usepackage{bbm}
\usepackage{amscd}
\usepackage{amssymb}
\usepackage{amsthm}
\usepackage{amsrefs}
\usepackage{latexsym}
\usepackage{mathrsfs}
\usepackage{bbm}
\usepackage{enumerate}
\usepackage{graphicx}
\usepackage{color}
\usepackage{mathtools}
\begin{document}
\newtheorem{theorem}[subsection]{Theorem}
\newtheorem{proposition}[subsection]{Proposition}
\newtheorem{lemma}[subsection]{Lemma}
\newtheorem{corollary}[subsection]{Corollary}
\newtheorem{conjecture}[subsection]{Conjecture}
\newtheorem{prop}[subsection]{Proposition}
\numberwithin{equation}{section}
\renewcommand{\thefootnote}{\fnsymbol{footnote}}
\newcommand{\dif}{\mathrm{d}}
\newcommand{\intz}{\mathbb{Z}}
\newcommand{\ratq}{\mathbb{Q}}
\newcommand{\natn}{\mathbb{N}}
\newcommand{\comc}{\mathbb{C}}
\newcommand{\rear}{\mathbb{R}}
\newcommand{\prip}{\mathbb{P}}
\newcommand{\uph}{\mathbb{H}}
\newcommand\psum{\mathop{\sum\nolimits^{\mathrlap{*}}}}
\newcommand\dsum{\mathop{\sum\nolimits^{\mathrlap{'}}}}

\title{Mean square values of Dirichlet $L$-functions associated to fixed order characters}

\author[C. C. Corrigan]{Chandler C. Corrigan} 
\address{School of Mathematics and Statistics, University of New South Wales, Sydney, NSW 2052, Australia}
\email{c.corrigan@student.unsw.edu.au}

\date{September, 2023}
\maketitle

\begin{abstract}
    The main purpose of this paper is to establish bounds on the second moment of $L\big(\tfrac{1}{2}+it,\chi\big)$, averaged over families of fixed order characters.  A discrete version of the main result is also stated, from which zero-density estimates pertaining to fixed order characters are derived.
\end{abstract}

\section{Introduction}

    The distribution of values of Dirichlet $L$-functions on the critical line has been studied extensively in the literature.  One result in particular is that of Bourgain \cite{zz28}, who showed that for any $\varepsilon>0$ we have
    \begin{equation}\label{bourgain}
        \zeta\big(\tfrac{1}{2}+it\big)\ll|t|^{\frac{13}{84}+\varepsilon},
    \end{equation}
    which improved on previous bounds given in \cite{zz26,zz80,zz81,zz104,zz105,zz106,zz107,zz180,zz184,zz34,zz35,lit,zz195,ivic}.  It was conjectured by Lindel\"of \cite{zz113}, however, that the $\tfrac{13}{84}$ in the exponent on the right-hand side of \eqref{bourgain} can be omitted.  Building on the method developed by Bombieri and Iwaniec in \cite{zz26}, Huxley and Watt \cite{ar1,ar3,ar4} have produced results similar to \eqref{bourgain} for Dirichlet $L$-functions on the critical line.  Their bounds are of the form $L\big(\tfrac{1}{2}+it,\chi\big)\ll q^{a+\varepsilon}|t|^{b+\varepsilon}$, where $a,b>0$ are fixed and $q$ is the conductor of $\chi$.  Recently, the Weyl-bound $a,b=\tfrac{1}{6}$ was established by Petrow and Young \cite{pet1,pet2}, though, as is conjectured for the $\zeta$-function, it is hypothesised that we should be able to take $a,b=0$.\newline

    In order to obtain results which are sharper in the $q$-aspect, Payley \cite{payley} considered the mean square value of the $L$-functions associated to Dirichlet characters modulo $q$.  Slight refinements to his result have since been made in \cite{elliott7,gal2,moto1,zh1,zh2,zh3}, with the case $t=0$ receiving particular attention in \cite{art1}.  Jutila \cite{jut99} established similar bounds for real characters, and Baier and Young \cite{baieryoung} showed that
    \begin{equation}\label{baiyou}
        \sum_{\chi\in\mathcal{O}_3(Q)}\big|L\big(\tfrac{1}{2}+it,\chi\big)\big|^2\ll Q^{\frac{6}{5}+\varepsilon}(|t|+2)^{\frac{6}{5}+\varepsilon},
    \end{equation}
    where $\mathcal{O}_j(Q)$ denotes the set of primitive Dirichlet characters of order $j$ and conductor $q\in(Q,2Q]$.  Following the approach of Baier and Young, Gao and Zhao \cite{gaozhao} were able to establish the stronger result 
    \begin{equation}\label{gaozho}
        \sum_{\chi\in\mathcal{O}_4(Q)}\big|L\big(\tfrac{1}{2}+it,\chi\big)\big|^2\ll Q^{\frac{7}{6}+\varepsilon}(|t|+2)^{\frac{1}{2}+\varepsilon}
    \end{equation}
    pertaining to quartic characters.\newline

    In view of what is asserted by the generalised Lindel\"of hypothesis, we should expect that \eqref{baiyou} and \eqref{gaozho} are weak in the $t$-aspect.  In order to gain strength in this aspect, we can average over all $t$ in a certain range.  One of the earliest results in this direction is due to Hardy and Littlewood \cite{harlit}, who demonstrated that
    \begin{equation}\label{zeta2nd}
        \int\limits_{0}^{T}\left|\zeta\left(\tfrac{1}{2}+it\right)\right|^2\:\dif t \sim T\log T
    \end{equation}
    as $T\to\infty$.  Improvements to the error term were later made in \cite{lit,ing,tit1,tit2,atkinson}.  Gallagher \cite{gal2} and Motohashi \cite{moto2} gave expressions similar to \eqref{zeta2nd} pertaining to individual Dirichlet $L$-functions.  Their results, however, are not unconditionally strong in the $q$-aspect, since they do not average over a family of characters.  Watt \cite{wa1} gave a result which saves in the $q$-aspect, though his result is only applicable to averages taken over small segments of the critical line.  Averaging over all Dirichlet characters modulo $q$, Ramachandra \cite{rama2} derived a result which is sharp in the $q$-aspect.  Improvements to the error term of Ramachandra's result have been made in \cite{vvrane,meurman,balarama}.  For real characters, Jutila \cite{jut1} showed that
    \begin{equation}\label{realsq}
        \sum_{\chi\in\mathcal{O}_2(Q)}\int\limits_{-T}^{T}\big|L\big(\tfrac{1}{2}+it,\chi\big)\big|^2\:\dif t\ll(QT)^{1+\varepsilon},
    \end{equation}
    however there are no non-trivial results available for higher order characters.  Slightly altering the method used to establish \eqref{baiyou}, we can derive the following result which improves on \eqref{baiyou} and \eqref{gaozho} in the $T$-aspect, on average.
    \begin{theorem}\label{main1}
        Suppose that $Q\geqslant2$ and $T\gg Q^\frac{1}{5}$.  Then for $j=3,4,6$ and any $\varepsilon>0$ we have
        \begin{equation*}
            \sum_{\chi\in\mathcal{O}_j(Q)}\int\limits_{-T}^{T}\big|L\big(\tfrac{1}{2}+it,\chi\big)\big|^2\:\dif t\ll (QT)^{1+\varepsilon},
        \end{equation*}
        where the implied constants depend on $\varepsilon$ alone.
    \end{theorem}
    We may use a classical lemma of Gallagher \cite{gallgs} to replace the integral in Theorem~\ref{main1} with a finite sum over well-spaced points.  Consequently, we have the following discrete version of Theorem~\ref{main1}.
    \begin{corollary}\label{main2}
        Let $Q,T$ be as in Theorem~\ref{main1} and fix $\delta>0$.  For each $\chi\in\mathcal{O}_j$, suppose that $\mathcal{T}_\chi$ is a finite set of real numbers in $\big[\tfrac{1}{2}\delta-T,T-\tfrac{1}{2}\delta\big]$ such that $|t-t'|\geqslant\delta$ for all distinct $t,t'\in\mathcal{T}_\chi$.  Then for $j=3,4,6$ and any $\varepsilon>0$ we have
        \begin{equation*}
            \sum_{\chi\in\mathcal{O}_j(Q)}\sum_{t\in\mathcal{T}_\chi}\big|L\big(\tfrac{1}{2}+it,\chi\big)\big|^2\ll\big(\delta^{-1}+1\big)(QT)^{1+\varepsilon},
        \end{equation*}
        where the implied constants depend on $\varepsilon$ alone.
    \end{corollary}  
    Similarly, using \eqref{realsq} in place of Theorem~\ref{main1}, we can deduce an analog of Corollary~\ref{main2} for $\mathcal{O}_2$.  Naturally, the restriction $T\gg Q^\frac{1}{5}$ is dropped for real characters, since it is not present in \eqref{realsq}.\newline
    
    An important application of the discrete form of power moments, is in understanding the distribution of zeros of Dirichlet $L$-functions.  In particular, Montgomery \cite{HM,HM2} used a fourth moment bound to derive a zero-density result pertaining to the set of all primitive Dirichlet characters with conductor not exceeding $Q$, and in doing so improved on earlier results of Bombieri \cite{bombls} and Vinogradov \cite{vinogradov}.  Later, Jutila \cite{jut2} showed that
    \begin{equation}\label{zde1}
        \sum_{\chi\in\mathcal{O}_2(Q)}N(\sigma,T,\chi)\ll(QT)^{\frac{7-6\sigma}{6-4\sigma}+\varepsilon},
    \end{equation}
    where $N(\sigma,T,\chi)$ denotes the number of zeros of $L(s,\chi)$ in the rectangle $R(\sigma,T)=[\sigma,1]+i[-T,T]$.  After developing the large sieve for real characters, Heath-Brown \cite{DRHB} improved on \eqref{zde1} in the $Q$-aspect, however in \cite{cozh} we showed that
    \begin{equation}\label{zde2}
        \sum_{\chi\in\mathcal{O}_2(Q)}N(\sigma,T,\chi)\ll(QT)^\varepsilon\min\Big(\big(Q^3T^4\big)^\frac{1-\sigma}{2-\sigma},(QT)^\frac{3(1-\sigma)}{\sigma}\Big),
    \end{equation}     
    which improves on Heath-Brown's result for all $\sigma\in\big(\tfrac12,1\big)$.  We add to these estimates the following result, which is derived using the $\mathcal{O}_2$ analog of Corollary~\ref{main2}.
    \begin{theorem}\label{main3}
        Take $Q,T\geqslant2$, and suppose that $\sigma\in\big(\tfrac{1}{2},1\big)$.  Then for any $\varepsilon>0$ we have
        \begin{equation}\label{main3a}
            \sum_{\chi\in\mathcal{O}_2(Q)}N(\sigma,T,\chi)\ll(QT)^\varepsilon\min\Big((QT)^{\frac{4(1-\sigma)}{3-2\sigma}},\big(Q^4T^3\big)^{(1-\sigma)}\Big),
        \end{equation}
        where the implied constant depends on $\varepsilon$ alone.
    \end{theorem}
    While \eqref{main3a} is weaker than \eqref{zde2} in the $Q$-aspect, it is stronger in the $T$-aspect.  This is not surprising, since \eqref{zde2} was derived from a mean fourth power bound which was weak in the $T$-aspect.  We should also note that \eqref{main3a} is stronger than \eqref{zde1} for all $Q,T\geqslant2$.\newline
    
    The distribution of zeros for families of $L$-functions associated to higher order characters has received little attention.  The only results in this direction are contained in \cite{cozh}, where we showed that
    \begin{equation}\label{zde3}
        \sum_{\chi\in\mathcal{O}_3(Q)}N(\sigma,T,\chi)\ll(QT)^\varepsilon\min\Big(Q^\frac{125-108\sigma}{90-72\sigma}T^\frac{49-44\sigma}{22-8\sigma},(QT)^\frac{7(1-\sigma)}{2\sigma}\Big)
    \end{equation}
    when $T\gg Q^\frac{2}{3}$, and
    \begin{equation}\label{zde4}
        \sum_{\chi\in\mathcal{O}_4(Q)}N(\sigma,T,\chi)\ll(QT)^\varepsilon\min\Big(Q^\frac{41-36\sigma}{30-24\sigma}T^\frac{49-44\sigma}{22-8\sigma},(QT)^\frac{7(1-\sigma)}{2\sigma}\Big)
    \end{equation}
    when $T\gg Q^\frac{1}{2}$.  Similarly to \eqref{zde2}, these results were derived using a bound on the fourth moment which was weak in the $T$-aspect.  Following the method used to establish Theorem~\ref{main3}, we can derive the following result, which improves on \eqref{zde3} and \eqref{zde4} in the $T$-aspect.
    \begin{theorem}\label{main33}
        Suppose that $Q,T\geqslant2$ are such that $T\gg Q^\frac{1}{5}$.  Then for any $\varepsilon>0$ we have
        \begin{equation}\label{main3b}
            \sum_{\chi\in\mathcal{O}_3(Q)}N(\sigma,T,\chi)\ll(QT)^\varepsilon\min\Big(Q^\frac{16-10\sigma}{9}T^\frac{4(1-\sigma)}{3-2\sigma},Q^\frac{16(1-\sigma)}{9-6\sigma}T^\frac{4(1-\sigma)}{3-2\sigma},\big(Q^4T^3\big)^{(1-\sigma)}\Big)
        \end{equation}
        and
        \begin{equation}\label{main3c}
            \sum_{\chi\in\mathcal{O}_4(Q)}N(\sigma,T,\chi)\ll(QT)^\varepsilon\min\Big(Q^\frac{5-3\sigma}{3}T^\frac{4(1-\sigma)}{3-2\sigma},Q^\frac{5(1-\sigma)}{3-2\sigma}T^\frac{4(1-\sigma)}{3-2\sigma},\big(Q^4T^3\big)^{(1-\sigma)}\Big),
        \end{equation}
        where the implied constants depend on $\varepsilon$ alone.  Also, \eqref{main3b} holds if $\mathcal{O}_3$ is replaced by $\mathcal{O}_6$.
    \end{theorem}
    We note that \eqref{main3b} improves on \eqref{zde3}, and \eqref{main3c} on \eqref{zde4}, in the $Q$-aspect whenever $\sigma\leqslant\tfrac{7}{8}$.  The fact that \eqref{main3b} holds for sextic characters as well as cubic characters is due to a result of Baier and Young \cite{baieryoung}.

\section{Strategy}

    As is common for problems like Theorem~\ref{main1}, our starting point will be with an inequality of large sieve type (cf.\:\cite[\S9.6]{opera}).  Our results are heavily dependant on bounds available for the polynomials $D_j(Q,N)$ such that   
    \begin{equation}\label{bb}
        \sum_{\chi\in\mathcal{O}_j(Q)}\Big|\dsum_{N<n\leqslant2N}a_n\chi(n)\Big|^2\ll(QN)^\varepsilon D_j(Q,N)\dsum_{N<n\leqslant2N}|a_n|^2
    \end{equation}
    holds for any sequence $(a_n)$ of complex numbers, where the $'$ denotes that the sum is taken over square free natural numbers.  The case $j=2$ was originally studied by Heath-Brown \cite{DRHB}, with the cases $j=3$ and $j=4$ being studied by Baier and Young \cite{baieryoung} and Gao and Zhao \cite{gaozhao,gaozhaoold}, respectively.  Using these results in conjunction with \cite[Lemma 2.4.1]{corrigan}, we can give bounds on the $\Delta_j(Q,T,N)$ such that
    \begin{equation}\label{taspect}
        \sum_{\chi\in\mathcal{O}_j(Q)}\int\limits_{-T}^{T}\Big|\dsum_{N<n\leqslant2N}a_n\chi(n)n^{-it}\Big|^2\:\dif t\ll(QT)^\varepsilon\Delta_j(Q,T,N)\dsum_{N<n\leqslant2N}|a_n|^2,
    \end{equation}
    for any $\varepsilon>0$, where the implied constant depends on $\varepsilon$ alone.  How $\Delta_j(Q,T,N)$ is related to $D_j(Q,N)$ is outlined in \cite[\S2.4]{corrigan} and \cite[\S2]{cozh}, but as we will not be using $D_j(Q,N)$ directly, we will only state $\Delta_j(Q,T,N)$ here.  For real characters, we derive from \cite[Corollary 1]{DRHB} that
    \begin{equation}\label{ls2}
        \Delta_2(Q,T,N)\ll QT+N,
    \end{equation}
    which is clearly the best possible bound.  For cubic characters, \cite[Theorem 1.4]{baieryoung} gives us
    \begin{equation}\label{ls3}
        \Delta_3(Q,T,N)\ll\min\Big(Q^\frac{5}{3}T+N,Q^\frac{4}{3}T+Q^\frac{1}{2}N,Q^\frac{11}{9}T+Q^\frac{2}{3}N,QT+Q^\frac{1}{3}N^\frac{5}{3}T^{-\frac{2}{3}}+N^\frac{12}{5}T^{-\frac{7}{5}}\Big),
    \end{equation}
    and for quartic characters, \cite[Lemma 2.10]{gaozhao} gives us
    \begin{equation}\label{ls4}
        \Delta_4(Q,T,N)\ll\min\Big(Q^\frac{3}{2}T+N,Q^\frac{5}{4}T+Q^\frac{1}{2}N,Q^\frac{7}{6}T+Q^\frac{2}{3}N,QT+Q^\frac{1}{3}N^\frac{5}{3}T^{-\frac{2}{3}}+N^\frac{7}{3}T^{-\frac{4}{3}}\Big).
    \end{equation}  
    Additionally, by \cite[Theorem 1.5]{baieryoung}, we see that \eqref{ls3} holds for sextic characters as well.  It was originally believed that \eqref{ls2} should hold for $\mathcal{O}_3$ and $\mathcal{O}_4$ as well, in which case \eqref{realsq} and Theorem~\ref{main3} would also hold with $\mathcal{O}_2$ replaced with $\mathcal{O}_3$ or $\mathcal{O}_4$.  However, recent heuristics by Dunn and Radziwi\l\l{ }\cite{dunrad} have put doubt on this.\newline
    
    In order to prove Theorem~\ref{main1}, we will follow the general method that was used by Baier and Young \cite{baieryoung} to derive \eqref{baiyou}.  Using \eqref{taspect} in place of \eqref{bb} allows us to average over $t\in[-T,T]$, and consequently save in this aspect.  This also leads to improvements in the $T$-aspect of the zero-density estimates \eqref{zde2}, \eqref{zde3}, and \eqref{zde4}.
        
\section{Lemmata}   
    
    Using the approximate functional equation \cite[Theorem 5.3]{iwakow}, Baier and Young \cite{baieryoung} were able to establish the following result in the case $j=3$ and $\sigma=\tfrac{1}{2}$.
    \begin{lemma}\label{bylem}
        Let $Q,T\geqslant2$, and fix $j\geqslant2$ and $\sigma\in\mathbb{R}$ such that $\big|\sigma-\tfrac{1}{2}\big|\leqslant\tfrac{1}{2}(\log QT)^{-1}$.  Then for any $t\in[-T,T]$ and $\varepsilon>0$ we have
        \begin{equation*}
            \sum_{\chi\in\mathcal{O}_j(Q)}|L(\sigma+it,\chi)|^2\ll(QT)^\varepsilon\sum_{\chi\in\mathcal{O}_j(Q)}\sum_{\ell\leqslant\sqrt{2N}}\ell^{-1}\Big|\dsum_{N\ell^{-2}<n\leqslant2N\ell^{-2}}\chi(n)n^{-\frac{1}{2}-it}\Big|^2,
        \end{equation*}
        where $N=(QT)^{\frac{1}{2}+\varepsilon}$, and the implied constant depends on $\varepsilon$ alone.
    \end{lemma}
    \begin{proof}
        The result follows from the method outlined in \cite[\S4.3]{baieryoung}.
    \end{proof}
    In order to establish Corollary~\ref{main2}, we require the following result, which is originally due to Gallagher \cite{gallgs}.
    \begin{lemma}\label{b3}
        Take $T\geqslant\delta>0$, and let $\mathcal{T}\subset\big[\tfrac{1}{2}\delta-T,T-\tfrac{1}{2}\delta\big]$ be a finite collection of real numbers.  Define
        \begin{equation*}
            N_\delta(t)=\#\{t'\in\mathcal{T}:|t-t'|<\delta\}.
        \end{equation*}
        Then for any complex-valued function $f$, defined on $\big[\tfrac{1}{2}\delta-T,T-\tfrac{1}{2}\delta\big]$ with continuous derivative on $\big(\tfrac{1}{2}\delta-T,T-\tfrac{1}{2}\delta\big)$, we have
        \begin{equation*}
            \sum_{t\in\mathcal{T}}N_\delta(t)^{-1}|f(t)|^2\leqslant\delta^{-1}\int\limits_{-T}^{T}|f(t)|^2\:\dif t+\Big(\int\limits_{-T}^{T}|f(t)|^2\:\dif t\Big)^\frac{1}{2}\Big(\int\limits_{-T}^{T}|f'(t)|^2\:\dif t\Big)^\frac{1}{2},
        \end{equation*}
        where $f'$ denotes the first derivative of $f$ with respect to $t$.
    \end{lemma}
    \begin{proof}
        This is \cite[Lemma 1.4]{HM}.
    \end{proof}
    Our prerequisites for Corollary~\ref{main3} are the following two results pertaining to Dirichlet polynomials.
    \begin{lemma}\label{b5}
        Let $\mathcal{T}_\chi,\delta$ be as in Corollary~\ref{main2} and $Q,T,N,j\geqslant2$.  Fix $\sigma_0\in(0,1)$, and for each $\chi\in\mathcal{O}_j$, suppose that $\mathcal{S}_\chi=\{\sigma_t+it:t\in\mathcal{T}_\chi\}$ where $\sigma_t\in[\sigma_0,1)$ are arbitrarily chosen.  If $\mathcal{S}(Q)$ denotes the collection of all pairs $(s,\chi)$ where $s\in\mathcal{S}_\chi$ and $\chi\in\mathcal{O}_j(Q)$, then for any $\varepsilon>0$ we have    
        \begin{equation*}
            \sum_{(s,\chi)\in\mathcal{S}(Q)}\Big|\dsum_{n\leqslant N}a_n\chi(n)n^{-s}\Big|^2\ll\big(\delta^{-1}+1\big)(QTN)^\varepsilon\Delta_j(Q,T,N)\dsum_{n\leqslant N}|a_n|^2n^{-2\sigma_0},
        \end{equation*}
        where the implied constant depends on $\varepsilon$ alone.
    \end{lemma}
    \begin{proof}
        This is contained in \cite[Theorem 2.4.3]{corrigan}.
    \end{proof}
    \begin{lemma}\label{b4}
        Let $Q,T,N,j,\delta,\mathcal{S}(Q)$ be as in Lemma~\ref{b5}.  Then for any $\varepsilon>0$ we have    
        \begin{equation*}
            \sum_{(s,\chi)\in\mathcal{S}(Q)}\Big|\dsum_{n\leqslant N}a_n\chi(n)n^{-s}\Big|^2\ll\big(\delta^{-1}+1\big)(QTN)^\varepsilon\Big(N+QT^{\frac{1}{2}}|\mathcal{S}(Q)|\Big)\dsum_{n\leqslant N}|a_n|^2n^{-2\sigma_0},
        \end{equation*}
        where the implied constant depends on $\varepsilon$ alone.
    \end{lemma}
    \begin{proof}
        This is contained in \cite[Theorem 8.2]{HM}.
    \end{proof}    

\section{Proof of Theorem~\ref{main1} and Corollary~\ref{main2}}\label{proofsection}

    Suppose that $Q,T,\mathcal{T}_\chi,\delta$ are as in Corollary~\ref{main2}.  For $j=3,4,6$, we will prove for any $\varepsilon>0$ that
    \begin{equation}\label{prop0}
        \sum_{\chi\in\mathcal{O}_j(Q)}\int\limits_{-T}^{T}|L(\sigma+it,\chi)|^2\:\dif t\ll(QT)^\varepsilon\Delta_j\big(Q,T,(QT)^\frac{1}{2}\big),
    \end{equation}
    where $\sigma$ is such that $\big|\tfrac{1}{2}-\sigma\big|\leqslant\tfrac{1}{2}(\log QT)^{-1}$.  We will also prove that
    \begin{equation}\label{prop0c}
        \sum_{\chi\in\mathcal{O}_j(Q)}\sum_{t\in\mathcal{T}_\chi}\big|L\big(\tfrac{1}{2}+it,\chi\big)\big|^2\ll\big(\delta^{-1}+1\big)(QT)^\varepsilon\Delta_j\big(Q,T,(QT)^\frac{1}{2}\big)
    \end{equation}   
    for any $\varepsilon>0$.  The implied constants in \eqref{prop0} and \eqref{prop0c} depend on $\varepsilon$ alone.  On observing the last term in the minimum of \eqref{ls3} and \eqref{ls4}, we have
    \begin{equation*}
        \Delta_j\big(Q,T,(QT)^\frac{1}{2}\big)\ll QT
    \end{equation*}    
    for $j=3,4,6$ when $T\gg Q^\frac{1}{5}$, and consequently, we can derive Theorem~\ref{main1} from \eqref{prop0}.  We can derive Corollary~\ref{main2} in a similar manner from \eqref{prop0c}.  We will now prove \eqref{prop0} and \eqref{prop0c}, but note that we can derive slightly weaker results for $T\ll Q^\frac{1}{5}$.
    \begin{proof}[Proof of \eqref{prop0}]
        Fix $j\geqslant2$.  By Lemma~\ref{bylem} and \eqref{taspect} we have
        \begin{align*}
            \sum_{\chi\in\mathcal{O}_j(Q)}\int\limits_{-T}^{T}|L(\sigma+it,\chi)|^2\:\dif t&\ll(QT)^\varepsilon\sum_{\ell\leqslant\sqrt{2N}}\ell^{-1}\sum_{\chi\in\mathcal{O}_j(Q)}\int\limits_{-T}^{T}\Big|\dsum_{N\ell^{-2}<n\leqslant2N\ell^{-2}}\chi(n)n^{-\frac{1}{2}-it}\Big|^2\:\dif t\\
            &\ll(QT)^\varepsilon\sum_{\ell\leqslant\sqrt{2N}}\ell^{-1}\Delta_j\big(Q,T,N\ell^{-2}\big)\dsum_{N\ell^{-2}<n\leqslant2N\ell^{-2}}n^{-1}\\
            &\ll(QT)^\varepsilon\Delta_j(Q,T,N)\sum_{N<n\leqslant2N}n^{-1},
        \end{align*}
        where $N=(QT)^{\frac{1}{2}+\varepsilon}$.  This completes the proof.
    \end{proof}    
    \begin{proof}[Proof of \eqref{prop0c}]
        By Theorem~\ref{main1} and Lemma~\ref{b3} with Cauchy's inequality, we see that it suffices to show that
        \begin{equation}\label{d1}
            \sum_{\chi\in\mathcal{O}_j(Q)}\int\limits_{-T}^{T}\big|L'\big(\tfrac{1}{2}+it,\chi\big)\big|^2\:\dif t\ll(QT)^\varepsilon\Delta_j\big(Q,T,(QT)^\frac{1}{2}\big).
        \end{equation}
        Fix $t\in\mathbb{R}$ and $\chi\in\mathcal{O}_j(Q)$, and consider the circle $\mathcal{C}$ of radius $r=\tfrac{1}{2}(\log QT)^{-1}$ centred at $\tfrac{1}{2}+it$.  By Cauchy's differentiation formula we have
        \begin{equation*}
            \big|L'\big(\tfrac{1}{2}+it,\chi\big)\big|\ll(QT)^\varepsilon\int_{\mathcal{C}}|L(w,\chi)|\:|\dif w|,
        \end{equation*}
        and thus, after two applications of Cauchy's inequality, we obtain
        \begin{equation}\label{wamengti}
            \big|L'\big(\tfrac{1}{2}+it,\chi\big)\big|^2\ll(QT)^\varepsilon\int\limits_{t-r}^{t+r}\int\limits_{\frac{1}{2}-r}^{\frac{1}{2}+r}|L(\sigma+it,\chi)|^2\:\dif\sigma\:\dif t.
        \end{equation}
        Averaging \eqref{wamengti} over $\chi\in\mathcal{O}_j(Q)$ and $t\in[-T,T]$ we get
        \begin{align*}
            \sum_{\chi\in\mathcal{O}_j(Q)}\int\limits_{-T}^{T}\big|L'\big(\tfrac{1}{2}+it,\chi\big)\big|^2\:\dif t&\ll(QT)^\varepsilon\sum_{\chi\in\mathcal{O}_j(Q)}\int\limits_{-T-r}^{T+r}\int\limits_{\frac{1}{2}-r}^{\frac{1}{2}+r}|L(\sigma+it,\chi)|^2\:\dif\sigma\:\dif t\\&\ll(QT)^\varepsilon\max_{\frac{1}{2}-r\leqslant\sigma\leqslant\frac{1}{2}+r}\sum_{\chi\in\mathcal{O}_j(Q)}\int\limits_{-T-r}^{T+r}|L(\sigma+it,\chi)|^2\:\dif t\\
            &\ll(QT)^\varepsilon\Delta_j\big(Q,T,(QT)^\frac{1}{2}\big)
        \end{align*}
        by \eqref{prop0}, which proves \eqref{d1} and thus Corollary~\ref{main2}.
    \end{proof}
  
\section{Proof of Theorem~\ref{main3} and Theorem~\ref{main33}}   

    To prove our zero-density estimates, we follow the same general set up as in \cite[Theorem 12.2]{HM}.  Suppose that $X,Y\geqslant2$ are such that $X\ll Y\ll (QT)^K$ for some absolute constant $K$, define
    \begin{equation*}
        M_X(s,\chi)=\sum_{n\leqslant X}\mu(n)\chi(n)n^{-s}\quad\text{and}\quad\mathfrak{M}_X(s,\chi)=M_X(s,\chi)L(s,\chi)=\sum_n\mathfrak{m}_{X,n}\chi(n)n^{-s}
    \end{equation*}
    for appropriate $s\in\mathbb{C}$, where $\mu$ is the M\"obius function.  Note that $|\mathfrak{m}_{X,n}|\leqslant\tau(n)$, where $\tau(n)$ is the number of divisors of $n$.  Fix $j\geqslant2$ and let $\mathcal{R}$ be a set of $(\varrho,\chi)\in R(\sigma,T)\times\mathcal{O}_j(Q)$ such that $L(\varrho,\chi)=0$ and $|\gamma-\gamma'|\geqslant3C\log QT$ for all distinct $(\varrho,\chi)$ and $(\varrho',\chi)$ in $\mathcal{R}$ and some sufficiently large absolute constant $C$.  Here we have denoted, as is custom, the real part of $\varrho$ as $\beta$, and the imaginary part as $\gamma$.  Now, as in \cite[\S2]{cozh}, we can show that 
    \begin{equation*}
       \sum_{\chi\in\mathcal{O}_j(Q)}N(\sigma,T,\chi)\ll(QT)^\varepsilon(\#\mathcal{R}_1+\#\mathcal{R}_2),
    \end{equation*}   
    where $\mathcal{R}_1$ is the subset of $(\varrho,\chi)\in\mathcal{R}$ such that    
    \begin{equation*}
        \Big|\sum_{X<n\leqslant Y^2}\mathfrak{m}_{X,n}\chi(n)n^{-\varrho}e^{-\frac{n}{Y}}\Big|\geqslant\tfrac{1}{6},
    \end{equation*}
    and $\mathcal{R}_2$ is the subset such that
    \begin{equation*}
        \frac{1}{2\pi}\Big|\int\limits_{-C\log QT}^{C\log QT}\mathfrak{M}_X\big(\tfrac{1}{2}+i(\gamma+u),\chi\big)\Gamma\big(\tfrac{1}{2}-\beta+iu\big)Y^{\frac{1}{2}-\beta+iu}\:\dif u\Big|\geqslant\tfrac{1}{6}.
    \end{equation*}        
    The first terms in \eqref{main3a}, \eqref{main3b}, and \eqref{main3c} now follow from the fact that
    \begin{equation}\label{propa}
        \#\mathcal{R}_1+\#\mathcal{R}_2\ll(QT)^\varepsilon\big((QT)^\frac{1}{2}Y^{\frac{1}{2}-\sigma}\Delta_j(Q,T,X)^\frac{1}{2}+\Delta_j(Q,T,X)X^{1-2\sigma}+\Delta_j(Q,T,Y)Y^{1-2\sigma}\big),
    \end{equation}
    which is an analog of \cite[Theorem 2.2]{cozh}.  For $j=2$, we take
    \begin{equation*}
        X=QT\quad\text{and}\quad Y=(QT)^\frac{2}{3-2\sigma}
    \end{equation*}
    in \eqref{propa}, which allows us to obtain the first term in \eqref{main3a}.  Considering the third expression in the minimums \eqref{ls3} and \eqref{ls4}, for the case $j=3$ we take
    \begin{equation*}
        X=Q^\frac{5}{9}T\quad\text{and}\quad Y=Q^\frac{5}{9}T^\frac{2}{3-2\sigma}
    \end{equation*}
    in \eqref{propa}, which proves the first term in \eqref{main3b}, and for the case $j=4$ we take
    \begin{equation*}
        X=Q^\frac{1}{2}T\quad\text{and}\quad Y=Q^\frac{1}{2}T^\frac{2}{3-2\sigma}
    \end{equation*}
    in \eqref{propa}, which proves the first term in \eqref{main3c}.  To obtain the second terms in \eqref{main3b} and \eqref{main3c}, we consider the first expression in the minimums \eqref{ls3} and \eqref{ls4}, and take
    \begin{equation*}
        X=Q^\frac{5}{3}T\quad\text{and}\quad Y=Q^\frac{8}{9-6\sigma}T^\frac{2}{3-2\sigma}
    \end{equation*}
    in the case $j=3$, and
    \begin{equation*}
        X=Q^\frac{3}{2}T\quad\text{and}\quad Y=Q^\frac{5}{6-4\sigma}T^\frac{2}{3-2\sigma}
    \end{equation*}
    in the case $j=4$.  Note that the condition $X\ll Y$ is satisfied by the above choices precisely when the second terms in \eqref{main3b} and \eqref{main3c} are smaller than the respective first terms.  So, we are now only required to derive \eqref{propa}, and this we shall do presently.
    \begin{proof}[Proof of \eqref{propa}] 
        Using standard techniques, we can show that there exists a $U\geqslant2$ with $X\leqslant U\leqslant Y^2$ such that    
        \begin{equation}\label{e2}
            \#\mathcal{R}_1\ll(QT)^\varepsilon\sum_{\ell\leqslant\sqrt{2U}}\ell^{-1}\sum_{(\varrho,\chi)\in\mathcal{R}_1}\Big|\dsum_{U\ell^{-2}\leqslant k\leqslant 2U\ell^{-2}}\mathfrak{m}_{X,k\ell^2}\chi(k)k^{-\varrho}e^{-\frac{k\ell^2}{Y}}\Big|^2.
        \end{equation}
        Arguing along similar lines to \cite[(12.34)]{HM}, we can show that every $(\varrho,\chi)\in\mathcal{R}_2$ satisfies
        \begin{equation}\label{e3}
            Y^{\frac{1}{2}-\sigma}\log QT\big|\mathfrak{M}_X\big(\tfrac{1}{2}+it_\varrho,\chi\big)\big|\gg1,
        \end{equation}
        where $t_\varrho$ is the real number with $|t_\varrho-\gamma|\leqslant C\log QT$ which maximises $\big|\mathfrak{M}_X\big(\tfrac{1}{2}+it_\varrho,\chi\big)\big|$.  By \eqref{e3} and Cauchy's inequality we have
        \begin{equation*}
            \#\mathcal{R}_2\ll(QT)^\varepsilon Y^{\frac{1}{2}-\sigma}\Big(\sum_{(\varrho,\chi)\in\mathcal{R}_2}\big|M_X\big(\tfrac{1}{2}+it_\varrho,\chi\big)\big|^2\Big)^\frac{1}{2}\Big(\sum_{(\varrho,\chi)\in\mathcal{R}_2}\big|L\big(\tfrac{1}{2}+it_\varrho,\chi\big)\big|^2\Big)^\frac{1}{2},
        \end{equation*}
        and thus \eqref{propa} follows from Corollary~\ref{main2}, Lemma~\ref{b5}, and \eqref{e2} with $\delta\asymp\log QT$.
    \end{proof}    
    To obtain the final term in \eqref{main3a}, \eqref{main3b}, and \eqref{main3c}, it clearly suffices to show that
    \begin{equation}\label{propb}
        \#\mathcal{R}_1+\#\mathcal{R}_2\ll Q^{4(1-\sigma)+\varepsilon}T^{3(1-\sigma)+\varepsilon}
    \end{equation}   
    for appropriate $Q$ and $T$.  The proof of \eqref{propb} goes much the same as \cite[Theorem 2.3]{cozh}, but there are some differences that we will now illustrate.
    \begin{proof}[Proof of \eqref{propb}]
        Fix $j\in\{2,3,4\}$, and if $j\neq2$ then suppose that $T\gg Q^\frac{1}{5}$.  If we require $X$ to be large enough that it satisfies
        \begin{equation}\label{e4}
            X^{2\sigma-1}\geqslant C_1Q^{1+\varepsilon}T^{\frac{1}{2}+\varepsilon}
        \end{equation}
        for some sufficiently large constant $C_1>0$, then by \eqref{e2} and Lemma~\ref{b4} we have
        \begin{equation}\label{e5}
            \#\mathcal{R}_1\ll(QT)^\varepsilon\big(Y^{2-2\sigma}+\#\mathcal{R}_1QT^\frac{1}{2}X^{1-2\sigma}\big)\ll(QT)^\varepsilon Y^{2-2\sigma}.
        \end{equation}
        Now, for any $V\geqslant1$, by \eqref{realsq} or Corollary~\ref{main2} there are $O\big(Q^{1+\varepsilon}T^{1+\varepsilon}V^{-2}\big)$ pairs $(\varrho,\chi)\in\mathcal{R}_2$ such that $\big|L\big(\tfrac{1}{2}+it_\varrho,\chi\big)\big|\geqslant V$.  For the $(\varrho,\chi)\in\mathcal{R}_2$ that do not satisfy this condition, by \eqref{e3} we have 
        \begin{equation}\label{e6}
            VY^{\frac{1}{2}-\sigma}\log QT\big|M_X\big(\tfrac{1}{2}+it_\varrho,\chi\big)\big|\gg1.
        \end{equation}
        Let $\mathcal{R}_2'$ be the subset of $(\varrho,\chi)\in\mathcal{R}_2$ which satisfy \eqref{e6}.  If we require $Y$ to be large enough that it satisfies
        \begin{equation}\label{e7}
            Y^{2\sigma-1}\geqslant C_2V^2Q^{1+\varepsilon}T^{\frac{1}{2}+\varepsilon}
        \end{equation}
        for some sufficiently large $C_2>0$, then by \eqref{e6} and Lemma~\ref{b4} we have
        \begin{equation*}
            \#\mathcal{R}_2'\ll(QT)^\varepsilon V^2Y^{1-2\sigma}\big(X+\#\mathcal{R}_2'QT^{\frac{1}{2}}\big)\ll(QT)^\varepsilon V^2XY^{1-2\sigma},
        \end{equation*}
        and consequently we have
        \begin{equation}\label{e8}
            \#\mathcal{R}_2\ll(QT)^\varepsilon\big(QTV^{-2}+V^2XY^{1-2\sigma}\big).
        \end{equation}
        Taking $X$ and $Y$ so that equalities hold in \eqref{e4} and \eqref{e7}, we see from \eqref{e5} and \eqref{e8} that
        \begin{equation*}
            \#\mathcal{R}_1+\#\mathcal{R}_2\ll(QT)^\varepsilon\big(QTV^{-2}+\big(Q^2T+Q^2TV^4\big)^\frac{1-\sigma}{2\sigma-1}\big),
        \end{equation*}
        from which the desired result follows on taking $V=Q^\frac{4\sigma-3}{2}T^\frac{3\sigma-2}{2}$.
    \end{proof}    

\section{Higher power moments}

    Ingham \cite{ing} gave an expression for the fourth moment of the $\zeta$-function, analogous to \eqref{zeta2nd}.  Slight improvements to the error term of Ingham's result were made in \cite{tit3,tit10,atkinson2,bellman10,hb4,hb5,hb6}.  Haselgrove \cite{haselgrove} and Topacogullari \cite{topa} gave bounds pertaining to the mean fourth power of an individual Dirichlet $L$-function on the critical line, however their results are not unconditionally strong in the $q$-aspect.\newline

    Bounds for the mean fourth power of $L\big(\tfrac{1}{2}+it,\chi\big)$ over $\chi\bmod{q}$ were established in \cite{payley,yuvlinnik,gallgs} for fixed $t$, with the case $t=0$ receiving particular attention in \cite{abc1,abc2}.  Mean fourth power results pertaining to fixed order characters have been published in \cite{cozh,DRHB}, and higher power moments have been studied in \cite{ab1,ab2,ab3,ab4}.  We add to these results the following proposition. 
    \begin{proposition}\label{prop1}
        Let $k\geqslant1$ be an integer, and $Q,T\geqslant2$ be real numbers.  Then for any $t\in[-T,T]$ and $\varepsilon>0$, we have
        \begin{equation*}
            \sum_{\chi\in\mathcal{O}_2(Q)}\big|L\big(\tfrac{1}{2}+it,\chi\big)\big|^{2k}\ll(QT)^{\frac{1}{2}k+\varepsilon},
        \end{equation*}
        where the implied constant depends on $k$ and $\varepsilon$ at most.
    \end{proposition}  
    \begin{proof}
        Using a Mellin transform, we may show that
        \begin{equation*}
            L(s,\chi)^k=\sum_n\tau_k(n)\chi(n)n^{-s}e^{-\frac{n}{U}}-\frac{1}{2\pi i}\int\limits_{(c)}L(w,\chi)^k\:\Gamma(w-s)U^{w-s}\:\dif w
        \end{equation*}
        for appropriate $c$, where $\tau_k(n)$ denotes the number of ways to write $n$ as a product of $k$ factors.  The assertion then follows in a similar manner to \cite[Theorem 2]{DRHB}.
    \end{proof}     
    Analogs of Proposition~\ref{prop1} can be derived for higher order characters, however they only hold when $T\gg Q$.  Additionally, the exponent on the right-hand side of the above is expected to be far from sharp.\newline
    
    The mean fourth power of $L\big(\tfrac{1}{2}+it,\chi\big)$ on the critical line was studied in \cite{HM,rane,buihb,wwang}.  Unfortunately, using the same method of proof as in \S\ref{proofsection} to average over the $t\in[-T,T]$ does not produce a $2k$ analog any stronger than what we can obtain as a trivial consequence of Proposition~\ref{prop1}.  Under the assumption that the Lindel\"of hypothesis holds on-average for $L$-functions associated to real characters, we formulate the following.
    \begin{conjecture}\label{conj1}
        Let $k,Q,T$ be as in Proposition~\ref{prop1}.  Then for any $\varepsilon>0$ we have
        \begin{equation*}
            \sum_{\chi\in\mathcal{O}_2(Q)}\int\limits_{-T}^{T}\big|L\big(\tfrac{1}{2}+it,\chi\big)\big|^{2k}\:\dif t\ll(QT)^{1+\varepsilon},
        \end{equation*}
        where the implied constant depends on $k$ and $\varepsilon$ at most.
    \end{conjecture}
    We can derive a discrete equivalent of Conjecture~\ref{conj1} in the same manner in which we derived Corollary~\ref{main2} from Theorem~\ref{main1}.  The following improvement of Theorem~\ref{main3} then follows immediately.
    \begin{conjecture}
        Take $Q,T\geqslant2$, and suppose that $\sigma\in\big(\tfrac{1}{2},1\big)$.  Then for any $\varepsilon>0$ we have
        \begin{equation*}
            \sum_{\chi\in\mathcal{O}_2(Q)}N(\sigma,T,\chi)\ll(QT)^{2(1-\sigma)+\varepsilon},
        \end{equation*}
        where the implied constant depends on $\varepsilon$ alone.
    \end{conjecture}
    The above is often referred to as the density conjecture for real characters.  Naturally, analogous conjectures can be made about higher order characters, however, for the sake of brevity, we do not write them down here.
    
\vspace*{.5cm}

\noindent{\bf Acknowledgments.}  The author would like to thank the University of New South Wales for access to some of the resources that were necessary to complete this paper. Thanks is also given to Dr.~Liangyi~Zhao who assisted the author in obtaining hard copies of \cite{opera} and \cite{iwakow}.

\bibliography{secondmoment}

\begin{bibdiv}
\begin{biblist}

\bib{atkinson2}{article}{
      author={Atkinson, F.~V.},
       title={The mean value of the zeta-function on the critical line},
        date={1941},
     journal={Proc. London Math. Soc.},
      volume={47},
       pages={174\ndash 200},
}

\bib{atkinson}{article}{
      author={Atkinson, F.~V.},
       title={The mean-value of the {R}iemann zeta function},
        date={1949},
     journal={Acta Math.},
      volume={81},
       pages={353\ndash 376},
}

\bib{baieryoung}{article}{
      author={Baier, S.},
      author={Young, M.~P.},
       title={Mean values with cubic characters},
        date={2010},
     journal={J. Number Theory},
      volume={130},
       pages={879\ndash 903},
}

\bib{balarama}{incollection}{
      author={Balasubramanian, R.},
      author={Ramachandra, K.},
       title={A hybrid version of a theorem of {I}ngham},
        date={1985},
   booktitle={Number theory ({O}otacamund, 1984)},
      series={Lecture Notes in Math.},
      volume={1122},
   publisher={Springer, Berlin},
       pages={38\ndash 46},
}

\bib{bellman10}{article}{
      author={Bellman, R.},
       title={Wigert's approximate functional equation and the {R}iemann zeta function},
        date={1949},
     journal={Duke Math. J.},
      volume={16},
       pages={547\ndash 552},
}

\bib{bombls}{article}{
      author={Bombieri, E.},
       title={On the large sieve},
        date={1965},
     journal={Mathematika},
      volume={12},
       pages={201\ndash 225},
}

\bib{zz26}{article}{
      author={Bombieri, E.},
      author={Iwaniec, H.},
       title={On the order of $\zeta\big(\tfrac{1}{2}+it\big)$},
        date={1986},
     journal={Ann. Scuola Norm. Sup. Pisa Cl. Sci.},
      volume={13},
       pages={449\ndash 472},
}

\bib{zz28}{article}{
      author={Bourgain, J.},
       title={Decoupling, exponential sums, and the {R}iemann zeta function},
        date={2017},
     journal={J. Amer. Math. Soc.},
      volume={30},
       pages={205\ndash 224},
}

\bib{buihb}{article}{
      author={Bui, H.~M.},
      author={Heath-Brown, D.~R.},
       title={A note on the fourth moment of {D}irichlet ${L}$-functions},
        date={2010},
     journal={Acta Arith.},
      volume={141},
       pages={335\ndash 344},
}

\bib{ab1}{article}{
      author={Bui, H.~M.},
      author={Keating, J.~P.},
       title={On the mean values of {D}irichlet ${L}$-functions},
        date={2007},
     journal={Proc. London Math. Soc.},
      volume={95},
       pages={273\ndash 298},
}

\bib{corrigan}{misc}{
      author={Corrigan, C.~C.},
       title={On the distribution of zeros for families of {D}irichlet ${L}$-functions},
 institution={The University of New South Wales},
        date={2022},
        note={(Honours Thesis)},
}

\bib{cozh}{article}{
      author={Corrigan, C.~C.},
      author={Zhao, L.},
       title={Zero density theorems for families of {D}irichlet ${L}$-functions},
        date={2023},
     journal={Bull. Aust. Math. Soc.},
      volume={108},
       pages={224\ndash 235},
}

\bib{dunrad}{article}{
      author={Dunn, A.},
      author={Radziwi\l\l, M.},
       title={Bias in cubic {G}auss sums: {P}atterson’s conjecture},
        date={preprint},
        note={arXiv:2109.07463 [math.NT]},
}

\bib{elliott7}{article}{
      author={Elliott, P. D. T.~A.},
       title={On the distribution of the values of {D}irichlet ${L}$-series in the half-plane $\sigma>\tfrac{1}{2}$},
        date={1971},
     journal={Ind. Math.},
      volume={33},
       pages={222\ndash 234},
}

\bib{opera}{book}{
      author={Friedlander, J.~B.},
      author={Iwaniec, H.},
       title={Opera de {C}ribro},
      series={Amer. Math. Soc. Coll. Publ.},
   publisher={Amer. Math. Soc., Providence, RI},
        date={2010},
      volume={57},
}

\bib{gallgs}{article}{
      author={Gallagher, P.~X.},
       title={The large sieve},
        date={1967},
     journal={Mathematika},
      volume={14},
       pages={14\ndash 20},
}

\bib{gal2}{article}{
      author={Gallagher, P.~X.},
       title={Local mean value and density estimates for {D}irichlet ${L}$-functions},
        date={1975},
     journal={Ind. Math. (Proc.)},
      volume={78},
       pages={259\ndash 264},
}

\bib{ab2}{article}{
      author={Gao, P.},
       title={Bounds for moments of {D}irichlet ${L}$-functions to a fixed modulus},
        date={2021},
        note={arXiv:2103.00149 [math.NT]},
}

\bib{ab3}{article}{
      author={Gao, P.},
       title={Sharp lower bounds for moments of quadratic {D}irichlet ${L}$-functions},
        date={2021},
        note={arXiv:2102.04027 [math.NT]},
}

\bib{ab4}{article}{
      author={Gao, P.},
       title={Sharp upper bounds for moments of quadratic {D}irichlet ${L}$-functions},
        date={2021},
        note={arXiv:2101.08483 [math.NT]},
}

\bib{gaozhaoold}{article}{
      author={Gao, P.},
      author={Zhao, L.},
       title={Large sieve inequalities for quartic character sums},
        date={2012},
     journal={Q. J. Math.},
      volume={63},
       pages={891\ndash 917},
}

\bib{gaozhao}{article}{
      author={Gao, P.},
      author={Zhao, L.},
       title={Moments of central values of quartic {D}irichlet {$L$}-functions},
        date={2021},
     journal={J. Number Theory},
      volume={228},
       pages={342\ndash 358},
}

\bib{harlit}{article}{
      author={Hardy, G.~H.},
      author={Littlewood, J.~E.},
       title={Contributions to the theory of the {R}iemann zeta-function and the theory of the distribution of primes},
        date={1918},
     journal={Acta Math.},
      volume={41},
       pages={119\ndash 196},
}

\bib{haselgrove}{article}{
      author={Haselgrove, C.~B.},
       title={Some theorems in the analytic theory of numbers},
        date={1951},
     journal={J. London Math. Soc.},
      volume={26},
       pages={273\ndash 277},
}

\bib{hb4}{article}{
      author={Heath-Brown, D.~R.},
       title={The fourth power moment of the {R}iemann zeta function},
        date={1979},
     journal={Proc. London Math. Soc.},
      volume={38},
       pages={385\ndash 422},
}

\bib{art1}{article}{
      author={Heath-Brown, D.~R.},
       title={An asymptotic series for the mean value of {D}irichlet ${L}$-functions},
        date={1981},
     journal={Comment. Math. Helv.},
      volume={56},
       pages={148\ndash 161},
}

\bib{abc1}{article}{
      author={Heath-Brown, D.~R.},
       title={The fourth power mean of {D}irichlet's ${L}$-functions},
        date={1981},
     journal={Analysis},
      volume={1},
       pages={25\ndash 32},
}

\bib{DRHB}{article}{
      author={Heath-Brown, D.~R.},
       title={A mean value estimate for real character sums},
        date={1995},
     journal={Acta Arith.},
      volume={72},
       pages={235\ndash 275},
}

\bib{zz80}{book}{
      author={Huxley, M.~N.},
       title={Integer points, exponential sums, and the {R}iemann zeta function, {N}umber theory for the millennium, {II} ({U}rbana, {IL}, 2000)},
   publisher={A K Peters, Natick, MA},
        date={2002},
}

\bib{zz81}{article}{
      author={Huxley, M.~N.},
       title={Exponential sums and the {R}iemann zeta function, {V}},
        date={2005},
     journal={Proc. London Math. Soc.},
      volume={90},
       pages={1\ndash 41},
}

\bib{ar3}{book}{
      author={Huxley, M.~N.},
      author={Watt, N.},
       title={Congruence families of exponential sums, {A}nalytic number theory},
      series={London Math. Soc. Lecture Note Series},
   publisher={Cambridge University Press},
        date={1997},
      volume={247},
}

\bib{ar1}{article}{
      author={Huxley, M.~N.},
      author={Watt, N.},
       title={Hybrid bounds for {D}irichlet's ${L}$-functions},
        date={2000},
     journal={Math. Proc. Camb. Phil. Soc.},
      volume={129},
       pages={385\ndash 415},
}

\bib{ing}{article}{
      author={Ingham, A.~E.},
       title={Mean-value theorems in the theory of the {R}iemann zeta-function},
        date={1927},
     journal={Proc. London Math. Soc.},
      volume={27},
       pages={273\ndash 300},
}

\bib{ivic}{article}{
      author={Ivi\'{c}, A.},
       title={Exponent pairs and the zeta function of {R}iemann},
        date={1980},
     journal={Studia Sci. Math. Hungar.},
      volume={15},
       pages={157\ndash 181},
}

\bib{hb6}{article}{
      author={Ivi\'{c}, A.},
      author={Motohashi, Y.},
       title={On the fourth moment of the {R}iemann zeta-function},
        date={1995},
     journal={J. Number Theory},
      volume={51},
       pages={16\ndash 45},
}

\bib{iwakow}{book}{
      author={Iwaniec, H.},
      author={Kowalski, E.},
       title={Analytic {N}umber {T}heory},
      series={Amer. Math. Soc. Coll. Publ.},
   publisher={Amer. Math. Soc., Providence, RI},
        date={2004},
      volume={53},
}

\bib{jut2}{article}{
      author={Jutila, M.},
       title={On mean values of {D}irichlet polynomials with real characters},
        date={1975},
     journal={Acta Arith.},
      volume={27},
       pages={191\ndash 198},
}

\bib{jut1}{article}{
      author={Jutila, M.},
       title={On mean values of {$L$}-functions and short character sums with real characters},
        date={1975},
     journal={Acta Arith.},
      volume={26},
       pages={405–410},
}

\bib{jut99}{article}{
      author={Jutila, M.},
       title={On the mean value of ${L}\big(\tfrac{1}{2},\chi\big)$ for real characters},
        date={1981},
     journal={Analysis},
      volume={1},
       pages={149\ndash 161},
}

\bib{zh2}{article}{
      author={Katsurada, M.},
       title={Asymptotic expansions of the mean values of {D}irichlet ${L}$-functions. {III}},
        date={1994},
     journal={Manuscripta Math.},
      volume={83},
       pages={425\ndash 442},
}

\bib{zh3}{article}{
      author={Katsurada, M.},
      author={Matsumoto, M.},
       title={Asymptotic expansions of the mean values of {D}irichlet ${L}$-functions},
        date={1991},
     journal={Math. Z.},
      volume={208},
       pages={23\ndash 39},
}

\bib{zz104}{article}{
      author={Kolesnick, G.},
       title={The improvement of the error term in the divisor problem},
        date={1969},
     journal={Mat. Zametki},
      volume={6},
       pages={545\ndash 554},
}

\bib{zz105}{article}{
      author={Kolesnick, G.},
       title={On the estimation of certain trigonometric sums},
        date={1973},
     journal={Acta Arith.},
      volume={25},
       pages={7\ndash 30},
}

\bib{zz106}{article}{
      author={Kolesnick, G.},
       title={On the order of $\zeta\big(\tfrac{1}{2}+it\big)$ and $\delta(r)$},
        date={1982},
     journal={Pacific J. Math.},
      volume={82},
       pages={115\ndash 143},
}

\bib{zz107}{article}{
      author={Kolesnick, G.},
       title={On the method of exponential pairs},
        date={1985},
     journal={Acta Arith.},
      volume={45},
       pages={115\ndash 143},
}

\bib{zz113}{article}{
      author={Lindel\"of, E.},
       title={Quelques remarques sur la croissance de la fonction $\zeta(s)$},
        date={1908},
     journal={Bull. Sci. Math.},
      volume={32},
       pages={341\ndash 356},
}

\bib{yuvlinnik}{article}{
      author={Linnik, Yu.~V.},
       title={All large numbers are the sums of two squares and a prime (a problem of {H}ardy and {L}ittlewood), {II}},
        date={1961},
     journal={Mat. Sb.},
      volume={53},
       pages={3\ndash 38},
}

\bib{lit}{article}{
      author={Littlewood, J.~E.},
       title={Researches in the theory of the {R}iemann $\zeta$-function},
        date={1922},
     journal={Proc. London Math. Soc.},
      volume={20},
       pages={xxii\ndash xxvii},
        note={Records},
}

\bib{meurman}{article}{
      author={Meurman, T.},
       title={A generalisation of {A}tkinson's formula to ${L}$-functions},
        date={1986},
     journal={Acta Arith.},
      volume={48},
       pages={351\ndash 370},
}

\bib{HM2}{article}{
      author={Montgomery, H.~L.},
       title={Zeros of {$L$}-functions},
        date={1969},
     journal={Invent. Math.},
      volume={8},
       pages={346–354},
}

\bib{HM}{book}{
      author={Montgomery, H.~L.},
       title={{T}opics in {M}ultiplicative {N}umber {T}heory},
      series={Lecture Notes in Mathematics},
   publisher={Spring-Verlag},
     address={Berlin},
        date={1971},
      volume={227},
}

\bib{moto1}{article}{
      author={Motohashi, Y.},
       title={A note on the mean value of the zeta and ${L}$-functions. {I}.},
        date={1985},
     journal={Proc. Japan Acad. Ser. A Math. Sci.},
      volume={61},
       pages={222\ndash 224},
}

\bib{moto2}{article}{
      author={Motohashi, Y.},
       title={A note on the mean value of the zeta and ${L}$-functions. {II}.},
        date={1985},
     journal={Proc. Japan Acad. Ser. A Math. Sci.},
      volume={61},
       pages={313–316},
}

\bib{payley}{article}{
      author={Payley, R. E. A.~C.},
       title={On the $k$-analogues of some theorems in the theory of the {R}iemann zeta-function},
        date={1931},
     journal={Proc. London Math. Soc.},
      volume={32},
       pages={273\ndash 311},
}

\bib{pet1}{article}{
      author={Petrow, I.},
      author={Young, M.~P.},
       title={The {W}eyl bound for {D}irichlet {$L$}-functions of cube-free conductor},
        date={2020},
     journal={Annals Math.},
      volume={192},
       pages={437\ndash 486},
}

\bib{pet2}{article}{
      author={Petrow, I.},
      author={Young, M.~P.},
       title={The fourth moment of {D}irichlet {$L$}-functions along a coset and the {W}eyl bound},
        date={2023},
     journal={Duke Math. J.},
      volume={172},
       pages={1879\ndash 1960},
}

\bib{rama2}{article}{
      author={Ramachandra, K.},
       title={Application of a theorem of {M}ontgomery and {V}aughan},
        date={1975},
     journal={J. London Math. Soc.},
      volume={10},
       pages={482\ndash 486},
}

\bib{vvrane}{article}{
      author={Rane, V.~V.},
       title={On the mean square value of {D}irichlet ${L}$-series},
        date={1980},
     journal={J. London Math. Soc.},
      volume={21},
       pages={203\ndash 215},
}

\bib{rane}{article}{
      author={Rane, V.~V.},
       title={A note on the mean value of ${L}$-series},
        date={1981},
     journal={Proc. Indian Acad. Sci. Math. Sci.},
      volume={90},
       pages={273\ndash 286},
}

\bib{abc2}{misc}{
      author={Soundararajan, K.},
       title={The fourth moment of {D}irichlet ${L}$-functions},
      series={Clay Math. Proc.},
   publisher={Amer. Math. Soc., Providence, RI},
        date={2007},
      volume={7},
}

\bib{tit3}{article}{
      author={Titchmarsh, E.~C.},
       title={The mean-value of the zeta-function on the critial line},
        date={1928},
     journal={Proc. London Math. Soc.},
      volume={27},
       pages={137\ndash 150},
}

\bib{zz180}{article}{
      author={Titchmarsh, E.~C.},
       title={On van der {C}orput's method and the zeta function of {R}iemann ({II})},
        date={1931},
     journal={Q. J. Math},
      volume={2},
       pages={313\ndash 320},
}

\bib{tit1}{article}{
      author={Titchmarsh, E.~C.},
       title={On van der {C}orput's method and the zeta function of {R}iemann ({III})},
        date={1932},
     journal={Q. J. Math.},
      volume={3},
       pages={133\ndash 141},
}

\bib{tit2}{article}{
      author={Titchmarsh, E.~C.},
       title={On van der {C}orput's method and the zeta function of {R}iemann ({V})},
        date={1934},
     journal={Q. J. Math.},
      volume={5},
       pages={195\ndash 210},
}

\bib{tit10}{article}{
      author={Titchmarsh, E.~C.},
       title={The mean value of $\left|\zeta\left(\tfrac{1}{2}+it\right)\right|^4$},
        date={1937},
     journal={Q. J. Math.},
      volume={8},
       pages={107\ndash 112},
}

\bib{zz184}{article}{
      author={Titchmarsh, E.~C.},
       title={On the order of $\zeta\big(\tfrac{1}{2}+it\big)$},
        date={1942},
     journal={Q. J. Math},
      volume={13},
       pages={11\ndash 17},
}

\bib{topa}{article}{
      author={Topacogullari, B.},
       title={The fourth moment of individual {D}irichlet ${L}$-functions on the critical line},
        date={2021},
     journal={Math. Z.},
      volume={298},
       pages={577–624},
}

\bib{zz34}{article}{
      author={van~der Corput, J.~G.},
       title={Versch\"arfung der {A}bsch\"atzung beim {T}eilerproblem},
        date={1922},
     journal={Math. Ann.},
      volume={89},
       pages={39\ndash 65},
}

\bib{zz35}{article}{
      author={van~der Corput, J.~G.},
       title={Zum {T}eilerproblem},
        date={1928},
     journal={Math. Ann.},
      volume={98},
       pages={697\ndash 716},
}

\bib{vinogradov}{article}{
      author={Vinogradov, A.~I.},
       title={On the density hypothesis for {D}irichlet {$L$}-series},
        date={1965},
     journal={Izv. Akad. Nauk SSSR Ser. Mat.},
      volume={29},
       pages={903–934},
}

\bib{wwang}{misc}{
      author={Wang, W.},
       title={{F}ourth power mean value of {D}irichlet’s ${L}$-functions, {I}nternational {S}ymposium in {M}emory of {H}ua {L}oo {K}eng: {V}olume {I} {N}umber theory},
   publisher={Springer, Berlin},
        date={1991},
}

\bib{ar4}{misc}{
      author={Watt, N.},
       title={A hybrid bound for {D}irichlet ${L}$-functions on the critical line, {P}roceedings of the {A}malfi {C}onference on {A}nalytic {N}umber {T}heory (1989)},
   publisher={University di Salerno},
        date={1992},
}

\bib{wa1}{article}{
      author={Watt, N.},
       title={On the mean squared modulus of a {D}irichlet ${L}$-function over a short segment of the critical line},
        date={2004},
     journal={Acta Arith.},
      volume={111},
       pages={307\ndash 403},
}

\bib{zz195}{article}{
      author={Weyl, H.},
       title={\"{U}ber die {G}leichverteilung von {Z}ahlen mod. {E}ins},
        date={1916},
     journal={Math. Ann.},
      volume={77},
       pages={313\ndash 352},
}

\bib{hb5}{misc}{
      author={Zavorotny\u{\i}, N.~I.},
       title={On the fourth moment of the {R}iemann zeta-function, {A}utomorphic functions and number theory, {I}},
   publisher={Akad. Nauk SSSR, Vladivostok},
        date={1989},
}

\bib{zh1}{article}{
      author={Zhang, W.},
       title={On the mean square value of {D}irichlet's ${L}$-functions},
        date={1992},
     journal={Comp. Math.},
      volume={84},
       pages={59\ndash 69},
}

\end{biblist}
\end{bibdiv}
\bibliographystyle{amsxport}

\end{document}